\documentclass[runningheads]{llncs}
\usepackage{amsmath,amssymb}
\usepackage{tikz}
\usepackage{graphicx}
\usepackage{mathrsfs}
\usepackage{colortbl} 
\usepackage{array}
\usepackage{bbding} 
\usepackage[colorlinks=true,linkcolor=blue,citecolor=blue,urlcolor=blue,breaklinks=true]{hyperref}

\title{Descriptive Complexity of Sensitivity of Cellular Automata}

\author{Tom Favereau \inst{1}(\Envelope) \orcidID{0009-0008-5876-8676} \and Ville Salo \inst{2} \orcidID{0000-0002-2059-194X}}
\institute{
  Mines Nancy, University of Lorraine, France \\
  \email{tom.favereau@etu.mines-nancy.univ-lorraine.fr} \and
  University of Turku, Finland \\
  \email{vosalo@utu.fi}
}

\begin{document}
\maketitle

\begin{abstract}
We study the computational complexity of determining whether a cellular automaton is sensitive to initial conditions. We show that this problem is $\Pi^0_2$-complete in dimension 1 and $\Sigma^0_3$-complete in dimension 2 and higher. This solves a question posed by Sablik and Theyssier. 
\end{abstract}

\section{Introduction}

Cellular automata are discrete dynamical systems that exhibit complex behavior despite their simple local rules. There have been numerous attempts to classify such systems, with the first notable classification proposed by Wolfram \cite{wolfram1983} for one-dimensional cellular automata. However, this classification lacked rigorous formality, which led K{\r{u}}rka~\cite{kurka1997} to propose a more formal classification of one-dimensional cellular automata based on their sensitivity to initial conditions.

Other classification schemes have also been proposed, such as Culik's classification. The arithmetical complexity of Culik's classes has been established and demonstrated by Sutner in~\cite{sutner1989}. The decidability of K{\r{u}}rka's classes, particularly the problem of sensitivity to initial conditions, was studied by Durand et al.~\cite{Durand2003}, while the reversible case was addressed by Lukkarila~\cite{lukkarila2010}.

Sablik and Theyssier~\cite{sablik2011topological} showed that K{\r{u}}rka's classification no longer holds in higher dimensions and that additional classes must be introduced. In the same article, they investigated the arithmetical complexity of K{\r{u}}rka's classes. They demonstrated that sensitivity to initial conditions is $\Pi^0_2$ in one dimension, but did not prove completeness, and showed $\Sigma^0_3$-hardness starting from dimension three, but the exact complexity of sensitivity within the arithmetical hierarchy was left open.

Understanding the dynamical properties of cellular automata, particularly their sensitivity to initial conditions, is crucial for characterizing their behavior. In this paper, we address these open questions and provide a precise characterization of the computational complexity of determining sensitivity across different dimensions. Specifically, we demonstrate that sensitivity to initial conditions is $\Pi^0_2$-complete in one dimension and $\Sigma^0_3$-complete in two dimensions and higher. As a corollary, we provide a new proof that the finite nilpotency problem for cellular automata, originally proven by Sutner in \cite{sutner1989}, is $\Pi^0_2$-complete. Additionally, we explore an application of our results by constructing a cellular automaton whose sensitivity is equivalent to the truth of the twin prime conjecture. These findings not only fill gaps left by previous research but also expose the fundamental difference in complexity between one-dimensional and higher-dimensional cellular automata with respect to sensitivity.

Our main results establish a complete classification of the complexity of the sensitivity problem in cellular automata. We formally state them as follows.

For one-dimensional cellular automata, establish:

\begin{theorem}\label{thm:dim1}
The problem of determining whether a one-dimensional cellular automaton is sensitive is $\Pi^0_2$-complete.
\end{theorem}

For higher dimension, we establish:

\begin{theorem}\label{thm:dimd}
For $d > 1$, the problem of determining whether a $d$-dimensional cellular automaton is sensitive is $\Sigma^0_3$-complete.
\end{theorem}

The proofs of these theorems rely on embedding Turing machines into cellular automata and reducing the problems TOT and COF to the respective sensitivity problem. With these new results, the classification of cellular automata with respect to sensitivity can be summarized in Table \ref{tab:classification}. Formal definitions and preliminaries are provided in Section \ref{sec:prel}. We present the proof of Theorem \ref{thm:dim1} in Section \ref{sec:proof1}, and the proof of Theorem \ref{thm:dimd} in Section \ref{sec:proof2}. An application is given in Section \ref{sec:twin}.

\begin{table}[h] 
\centering 
\caption{Descriptive complexity of K{\r{u}}rka's classification of cellular automata. The set $\mathcal{E}$ denotes the equicontinuity points. The notation $\Pi^0_2$-c indicates $\Pi^0_2$-completeness, while $\Sigma^0_1$-? signifies that completeness is unknown. A question mark (?) indicates that both completeness and the position in the hierarchy are unknown. Our contributions are highlighted with an asterisk.}
\resizebox{\textwidth}{!}{  
\begin{tabular}{|c|c|c|c|c|c|} 
\hline
 & $\mathcal E = A^{\mathbb Z^d}$ & $\mathcal E \neq \emptyset$ & $\mathcal E = \emptyset$ & Sensitive & Expansive \\
\hline
$d=1$       & $\Sigma^0_1$-c          &  $\Sigma^0_2$-c$^{*}$ & $\Pi^0_2$-c$^{*}$  & $\Pi^0_2$-c$^{*}$ & $\Sigma^0_1$-?   \\
\hline
$d>1$   &  $\Sigma^0_1$-c & ? & ?  &     $\Sigma^0_3$-c$^{*}$ & $\Sigma^0_1$-?      \\
\hline
\end{tabular}
}

\label{tab:classification} 
\end{table}

\section{Preliminaries}\label{sec:prel}

We will first recall some basic definitions that we will need throughout this document.

\begin{definition}[Cellular Automaton]
A \textbf{cellular automaton} is a quadruple $(d, S, N, f)$ where $d$ is the dimension of the automaton, $S$ is a finite set of states, $N$ is a finite subset of $\mathbb{Z}^d$ called the neighborhood, and $f: S^N \to S$ is the local transition function. This induces a global function $F: S^{\mathbb{Z}^d} \to S^{\mathbb{Z}^d}$.
\end{definition}

The sensitivity to initial conditions is a property of dynamical systems that states that for any configuration, we can always find an arbitrarily close configuration that tends to diverge from our original configuration. 
To formalize this we endow $S^{\mathbb{Z}^d}$ with the following metric:

\begin{equation}
d(x, y) = 2^{-\min\{\|n\|_\infty \mid x(n) \neq y(n)\}}
\end{equation}
where $ \|n\|_\infty = \max_{1 \le i \le d} |n_i| $ denotes the supremum norm on \( \mathbb{Z}^d \).
 
We formalize the concept of sensitivity in the following definition:

\begin{definition}[Sensitivity]
A cellular automaton $F$ is \textbf{sensitive} to initial conditions if there exists $\epsilon > 0$ such that for all configurations $x$ and $\delta > 0$, there exists $y \in B_\delta(x)$ and $n \in \mathbb{N}$ such that $d(F^n(x), F^n(y)) > \epsilon$.
\end{definition}

In one dimension, a famous result states a connection between sensitivity and having words that block information, i.e., blocking words. Blocking words are characterized by a length that no information can cross. Indeed, to truly prevent information from propagating, such words must have a length larger than the automaton's radius. 

Given a set $K \subset \mathbb{Z}^d$ and a word $w \in S^K$, we write $w \sqsubset_K u$ if $u \in S^{\mathbb{Z}^d}$ contains $w$ at position $K$, and we define the \emph{cylinder}: 
\begin{equation}
\text{Cyl}(w, K) = \{ u \in S^{\mathbb{Z}^d} \mid w \sqsubset_K u \}
\end{equation}

We provide a definition of an $m$-blocking word:

\begin{definition}[m-blocking word]
A word $w \in S^*$ together with an integer $p$ is \textbf{m-blocking} if for any configurations $u, v \in \text{Cyl}(w, [0, |w|-1])$, for all $n \in \mathbb{N}, F^n(u)_{|[p,p+m]} = F^n(v)_{|[p,p+m]}$.
\end{definition}
Hence, we can state the well-known characterization of sensitivity in one dimension:

\begin{theorem}[K{\r{u}}rka \cite{kurka2008}]
A one-dimensional cellular automaton with radius $r$ is not sensitive if and only if it has an $r$-blocking word.\end{theorem}

We will also need a notion of blocking word in higher dimensions. In dimension $d$, we want an $m$-blocking word to be any pattern that contains a subpattern of size $m^d$ that no information can enter.

We define the \emph{shift action} by:
\begin{align}
\sigma^p(x)_u &= x_{u+p}, \quad \forall x \in S^{\mathbb Z^d}, \ \forall u, p \in \mathbb Z^d \\
\sigma^p(K) &= \{u+p \mid u \in K\}, \quad \forall K \subset \mathbb Z^d, \ \forall p \in \mathbb Z^d
\end{align}

\begin{definition}[Higher dimensional m-blocking word]
For any finite subset $K \subset \mathbb{Z}^d$ together with a vector $p\in \mathbb {Z}^d$ such that $\sigma^p ([0, m]^d) \subseteq K$, a word $w \in S^K$ is \textbf{m-blocking} if for any configurations $u,v$ containing $w$ at the position $K$, for all $n \in \mathbb{N}$, $F^n(u)_{|\sigma^p([0, m]^d)} = F^n(v)_{|\sigma^p([0, m]^d)}$.
\end{definition}

With that definition, we state the following lemma that shows how one can extend the blocking word characterization for higher-dimensional cellular automata.

\begin{lemma}\label{lemma:blocking-words}
A cellular automaton is not sensitive if and only if it has arbitrarily large blocking words.
\end{lemma}

\begin{proof}
If $F$ is a sensitive cellular automaton, there exists $\epsilon > 0$ such that for all $x$ and $\delta > 0$, there exists $y \in B_{\delta}(x)$ and $n$ such that $d(F^n(x), F^n(y)) > \epsilon$. This means that there exists $m > -\log_2(\epsilon)$ such that arbitrarily close configurations will eventually differ on $[0, m]^d$, hence there is no $m$-blocking word.

Conversely, if $F$ is not sensitive, for every $\epsilon > 0$ and $m > -\log_2(\epsilon)$, we can find a pattern $\pi$ such that every configuration that matches with $\pi$ at position $0$ will never differ on $[0, m]^d$. Hence, we have arbitrarily large blocking words. \qed
\end{proof}

Turing machines are abstract computational models that provide a formal definition of a computable function. These machines can perform basic operations such as reading, writing, and moving the head. Their power lies in their ability to simulate any algorithm or computational process. We will use them to prove our complexity results.

A Turing machine is formally defined as a 5-tuple $(Q, A, q_i, q_h, \delta)$ where $Q$ is a finite set of states, including an initial state $q_i$ and a halting state $q_h$; $A$ is a finite alphabet that contains at least the blank symbol $\sqcup$ and the symbol $1$; and $\delta: Q \times A \to Q \times A \times \{-1, 0, 1\}$ is the transition function, such that for all $a \in A$, $\delta(q_h, a) = (q_h, a, 0)$. Inputs are natural numbers $n$ represented in unary as bi-infinite tapes containing only blanks and a block of $n$ consecutive $1$s, i.e., the tape content is $\sqcup^\infty 1^n \sqcup^\infty$, with the head initially positioned on the first $1$.

The transition function $\delta$ determines the machine's behavior: given a current state and a symbol under the head, it specifies the next state, the symbol to write, and the direction to move the head (-1 for left, 0 for stay, 1 for right).

Fix an enumeration of Turing machines; we denote by $M_e$ the $e$-th Turing machine. Let  $W_e := \{x \mid M_e \text{ halts on } x\}$ be the set of inputs on which $M_e$ halts. We now introduce the two sets that we will use in our proof:
\begin{equation}
    \text{TOT} := \{e \mid W_e = \mathbb{N}\}
\end{equation}
and 
\begin{equation}
    \text{COF} := \{e \mid W_e \text{ is cofinite}\}
\end{equation}
\begin{lemma}
    TOT is $\Pi^0_2$-complete and COF is $\Sigma^0_3$-complete. 
\end{lemma}

\begin{proof}
The reader can find a proof of these results in~\cite{rogers1987}. \qed
\end{proof}

Let us note a basic result about Turing machines: semi-infinite tape Turing machines are equivalent to two-way tape Turing machines.

\begin{proposition}
Given a Turing machine $M_2$, there exists a semi-infinite Turing machine $M_1$ such that for all $x \in \mathbb{N}$, $M_2$ halts on $x$ if and only if $M_1$ halts on $x$. 
\end{proposition}

\begin{proof}
The backward direction is trivial. For the converse, we can simulate any two-way Turing machine on a semi-infinite tape by instructing the machine, when it wants to move left to $0$, to copy all its non-blank symbols one step to the right. \qed
\end{proof}

In particular, this proposition allows us to consider the sets TOT and COF as sets of indices of semi-infinite Turing machines.  

Finally, we introduce our main tool for lifting results to higher dimensions: slicing.  

\begin{definition}[Slicing]  
For any $d$-dimensional cellular automaton with global function $F$, we can construct a $(d+1)$-dimensional cellular automaton with the same function $F$, which corresponds to applying the $d$-dimensional cellular automaton independently on each hyperplane $x_{d+1} = k$ (called a \textbf{slice}).  
\end{definition}  

This construction often allows us to reduce a $d$-dimensional problem to a $(d+1)$-dimensional one.


\section{One-dimensional Cellular Automata}\label{sec:proof1}

In this section, we prove our first result: the $\Pi^0_2$-completeness of the sensitivity problem for one-dimensional cellular automata. We begin with a simple lemma.

\begin{lemma}\label{lemma:blocking-word}
The problem of determining whether a given cellular automaton has an $m$-blocking word is in $\Sigma^0_2$.
\end{lemma}

\begin{proof}
Observe that in the definition of a blocking word, we only need to quantify over finite objects. Hence, having an $m$-blocking word can be written as:

\begin{equation}
\begin{split}
    \exists K \subset \mathbb{Z}^d, \exists p \in \mathbb{Z}^d, \exists w \in S^{K}, 
    \forall u, v \in (S^d)^*, n \in \mathbb{N}, \\
    (\sigma^p([0, m]^d) \subset K \land w \sqsubset_K u \land w \sqsubset_K v) \\
    \implies F^n(u)_{|\sigma^p([0, m]^d)} = F^n(v)_{|\sigma^p([0, m]^d)}
\end{split}
\end{equation}

This formulation clearly places the problem in $\Sigma^0_2$. \qed
\end{proof}

\subsection{Upper Bound on Complexity of the Sensitivity Problem}

The sensitivity problem for one-dimensional cellular automata is in $\Pi^0_2$. The proof follows directly from Lemma~\ref{lemma:blocking-word}. We note that this result was already known from Sablik and Theyssier \cite{sablik2011topological}.

\begin{lemma}\label{lemma:sensitivity-pi02}
The problem of determining whether a one-dimensional cellular automaton is sensitive to initial conditions is in $\Pi^0_2$.
\end{lemma}

\begin{proof}
A cellular automaton is sensitive to initial conditions if and only if it does not have a blocking word of any length. This is clearly a $\Pi^0_2$ formula, as it negates a $\Sigma^0_2$ formula (from Lemma~\ref{lemma:blocking-word}). \qed
\end{proof}

\subsection{Construction of $G_e$}

To prove $\Pi^0_2$-hardness, we reduce from TOT, which is known to be $\Pi^0_2$-hard. Given a number $e$, we denote the associated Turing machine by $M_e = (Q, A, q_i, q_h, \delta)$. 

It is important to note that we consider configurations starting from any state and with any distribution of symbols—not just those that result from a valid evolution of the Turing machine from its initial state. This includes configurations that do not correspond to any correct computation. Such 'degenerate' configurations arise because, in a cellular automaton, any configuration can serve as an initial one. We will explain how to address this issue in the following sections of the construction.

We construct a cellular automaton $G_e$ as follows:

The state set of $G_e$ is $(1 \cup \{\sqcup\}) \times ((Q \cup \{\sqcup\}) \times A) \times (\{<, >, p, \sqcup\} \cup X)$, where: $<$ and $>$ are delimiter symbols that partition the space into computational blocks and $X= \{x_r, x_l, x_0\}$ is a set of elements that handle the extension of the computational zone and restart the Turing machine on the input tape to prevent degenerate configurations.

It is a three-tape cellular automaton with a witness input tape, a working tape, and a delimiter tape.

The cellular automaton $G_e$ simulates multiple copies of $M_e$, where each copy operates within its own computational block, delimited by sequences of $<$ and $>$ symbols. A typical configuration is shown in Figure~\ref{fig:configuration}.

\begin{figure}[htbp!]
\centering
\begin{tikzpicture}[scale=0.7]
    \foreach \x in {0,...,3} {
        \draw (\x,0) rectangle (\x+1,1);
        \node at (\x+0.5,0.5) {$>$};
    }
    \draw (4,0) rectangle (5,1);
    \node at (4.5,0.5) {$x_l$};
    \foreach \x in {5,...,7} {
        \draw (\x,0) rectangle (\x+1,1);
        \node at (\x+0.5,0.5) {$<$};
    }
    
    \foreach \x in {8,...,9} {
        \draw (\x,0) rectangle (\x+1,1);
        \node at (\x+0.5,0.5) {$>$};
    }
    \draw (10,0) rectangle (11,1);
    \node at (10.5,0.5) {$x_r$};
    \foreach \x in {11,...,12} {
        \draw (\x,0) rectangle (\x+1,1);
        \node at (\x+0.5,0.5) {$<$};
    }
    
    \node at (-0.5,0.5) {$\cdots$};
    \node at (13.5,0.5) {$\cdots$};
    
    \foreach \x in {0,...,12} {
        \draw (\x,-1) rectangle (\x+1,0);
    }
    \node at (0.5,-0.5) {$a'$};
    \node at (1.5,-0.5) {$a'$};
    \node at (2.5,-0.5) {$b'$};
    \node at (3.5,-0.5) {$q, b'$};
    \node at (4.5,-0.5) {$c'$};
    \node at (5.5,-0.5) {$c'$};
    \node at (6.5,-0.5) {$d'$};
    \node at (7.5,-0.5) {$d'$};
    \node at (8.5,-0.5) {$a'$};
    \node at (9.5,-0.5) {$b'$};
    \node at (10.5,-0.5) {$q, c'$};
    \node at (11.5,-0.5) {$d'$};
    \node at (12.5,-0.5) {$a'$};
    
     \foreach \x in {0,...,12} {
        \draw (\x,-1) rectangle (\x+1,-2);
    }

    \node at (0.5,-1.5) {$a$};
    \node at (1.5,-1.5) {$a$};
    \node at (2.5,-1.5) {$b$};
    \node at (3.5,-1.5) {$b$};
    \node at (4.5,-1.5) {$c$};
    \node at (5.5,-1.5) {$c$};
    \node at (6.5,-1.5) {$d$};
    \node at (7.5,-1.5) {$d$};
    \node at (8.5,-1.5) {$a$};
    \node at (9.5,-1.5) {$b$};
    \node at (10.5,-1.5) {$c$};
    \node at (11.5,-1.5) {$d$};
    \node at (12.5,-1.5) {$a$};

     \node at (-0.5,-1.5) {$\cdots$};
    \node at (13.5,-1.5) {$\cdots$};
    \node at (-0.5,-0.5) {$\cdots$};
    \node at (13.5,-0.5) {$\cdots$};
\end{tikzpicture}
\caption{Structure of computational blocks in $G_e$. The upper row shows the delimiters and machine heads, while the lower rows show the tape content.}
\label{fig:configuration}
\end{figure}

The transition rules of $G_e$ enforce the following behaviors:

\begin{enumerate}
	\item The set $X$ contains $x_r$, which moves to the right on computational blocks, $x_l$, which moves to the left, and $x_0$, which moves to the left while resetting the computation on the witness tape. Symbols move back and forth in the computational zone and always try to extend it to the right when possible. When this happens, the Turing machine restarts on the input tape (eventually eliminating degenerate configurations).
	\item \texttt{<} and \texttt{>} delimit computational blocks. $x_r$ can only move to the right if it is to the left of a \texttt{<}; otherwise, it becomes an $x_l$. Similarly $x_0$ and $x_l$ can only move to the left if they are to the right of a \texttt{>}; otherwise, they become an $x_r$.
	\item The blank symbol on the witness tape propagates to the right on the computational block, preventing the input from being divided by a blank.  
	\item When a machine halts, it writes a blank symbol that propagates throughout the computational block.
	\item The $x$ symbol can extend its computational block if and only if the machine to its right has been destroyed.
	\item $p$ is a particle that always travels to the right on blank symbols of the first tape and gets destroyed when it encounters non-blank cells. It allows us to establish sensitivity.  
\end{enumerate}

We provide the complete rule of the cellular automaton $G_e$ in Figure~\ref{fig:transition1D}.

\begin{figure}[htbp!]
	\begin{center}
	
\begin{tikzpicture}

    \node (tab1) at (0,0) {
        \begin{tabular}{| >{\centering\arraybackslash}m{0.75cm}| >{\centering\arraybackslash}m{0.75cm}| >{\centering\arraybackslash}m{0.75cm}|}
        \hline
        \texttt{s} & \texttt{s} & \cellcolor{black}~ \\
        \hline
	\cellcolor{black}~ & q, b & \cellcolor{black} \\
        \hline
        \cellcolor{black} & \cellcolor{black} & \cellcolor{black} \\
        \hline
        \end{tabular}
    };

    \node (tab2) at (3, 0){
        \begin{tabular}{| >{\centering\arraybackslash}m{0.75cm}| >{\centering\arraybackslash}m{0.75cm}| >{\centering\arraybackslash}m{0.75cm}|}
        \hline
	\cellcolor{black}\\
	\hline
	b'\\
        \hline
          \cellcolor{black} \\
        \hline
        \end{tabular}
    };

    \draw[->, thick] (tab1) -- (tab2);
    \node at (1.5, 0.8) {$\delta_e(q, b) = (q', b',  -1)$};
    \node at (1.5, 1.2) {$s \in \{\texttt{<}, \texttt{>}\}$};
    
    \node (tab3) at (6,0) {
        \begin{tabular}{| >{\centering\arraybackslash}m{0.75cm}| >{\centering\arraybackslash}m{0.75cm}| >{\centering\arraybackslash}m{0.75cm}|}
        \hline
        \cellcolor{black} & \texttt{s} & \texttt{s}\\
        \hline
	\cellcolor{black} & q, b & \cellcolor{black} \\
        \hline
        \cellcolor{black} & \cellcolor{black} & \cellcolor{black} \\
        \hline
        \end{tabular}
    };

    \node (tab4) at (9,0) {
        \begin{tabular}{| >{\centering\arraybackslash}m{0.75cm}| >{\centering\arraybackslash}m{0.75cm}| >{\centering\arraybackslash}m{0.75cm}|}
        \hline
	\cellcolor{black} \\
        \hline
         b'  \\
        \hline
        \cellcolor{black} \\
        \hline
        \end{tabular}
    };
    \draw[->, thick] (tab3) -- (tab4);
    \node at (7.5, 0.8) {$\delta_e(q, b) = (q', b', 1)$};
     \node at (7.5, 1.2) {$s \in \{\texttt{<}, \texttt{>}\}$};

    \node (tab5) at (0,-2) {
        \begin{tabular}{| >{\centering\arraybackslash}m{0.75cm}| >{\centering\arraybackslash}m{0.75cm}| >{\centering\arraybackslash}m{0.75cm}|}
        \hline
        \cellcolor{black} & \texttt{s} & \texttt{s} \\
        \hline
	a & b & q, c \\
        \hline
        \cellcolor{black} & \cellcolor{black} & \cellcolor{black} \\
        \hline
        \end{tabular}
    };

    \node (tab6) at (3,-2) {
        \begin{tabular}{| >{\centering\arraybackslash}m{0.75cm}| >{\centering\arraybackslash}m{0.75cm}| >{\centering\arraybackslash}m{0.75cm}|}
        \hline
	\cellcolor{black} \\
        \hline
        q', b \\
        \hline
          \cellcolor{black} \\
        \hline
        \end{tabular}
    };
    \draw[->, thick] (tab5) -- (tab6);
    \node at (1.5, -1.2) {$\delta_e(q, c) = (q', c', -1)$};
     \node at (1.5, -0.8) {$s \in \{\texttt{<}, \texttt{>}\}$};
    
    \node (tab7) at (6,-2) {
        \begin{tabular}{| >{\centering\arraybackslash}m{0.75cm}| >{\centering\arraybackslash}m{0.75cm}| >{\centering\arraybackslash}m{0.75cm}|}
        \hline
        \texttt{s} & \texttt{s} & \cellcolor{black} \\
        \hline
	q, a & b & c \\
        \hline
        \cellcolor{black} & \cellcolor{black} & \cellcolor{black} \\
        \hline
        \end{tabular}
    };

    \node (tab8) at (9, -2){
        \begin{tabular}{| >{\centering\arraybackslash}m{0.75cm}| >{\centering\arraybackslash}m{0.75cm}| >{\centering\arraybackslash}m{0.75cm}|}
        \hline
        \cellcolor{black} \\
        \hline
	q', b \\
        \hline
          \cellcolor{black} \\
        \hline
        \end{tabular}
    };

    \draw[->, thick] (tab7) -- (tab8);
    \node at (7.5, -1.2) {$\delta_e(q, a) = (q', a', 1)$};
     \node at (7.5, -0.8) {$s \in \{\texttt{<}, \texttt{>}\}$};
    
    \node (tab9) at (0,-4) {
        \begin{tabular}{| >{\centering\arraybackslash}m{0.75cm}| >{\centering\arraybackslash}m{0.75cm}| >{\centering\arraybackslash}m{0.75cm}|}
        \hline
        \texttt{>} & $x_{l/0}$ & \cellcolor{black}\\
        \hline
	\cellcolor{black} & s & \cellcolor{black} \\
        \hline
        \cellcolor{black} & \cellcolor{black} & \cellcolor{black} \\
        \hline
        \end{tabular}
    };

    \node (tab10) at (3,-4) {
        \begin{tabular}{| >{\centering\arraybackslash}m{0.75cm}| >{\centering\arraybackslash}m{0.75cm}| >{\centering\arraybackslash}m{0.75cm}|}
        \hline
	\texttt{<} \\
        \hline
        \cellcolor{black} \\
        \hline
          \cellcolor{black} \\
        \hline
        \end{tabular}
    };
    \draw[->, thick] (tab9) -- (tab10);
    \node at (1.5, -3.2) {$s \neq (q_h, b)$};
    
    \node (tab11) at (6,-4) {
        \begin{tabular}{| >{\centering\arraybackslash}m{0.75cm}| >{\centering\arraybackslash}m{0.75cm}| >{\centering\arraybackslash}m{0.75cm}|}
        \hline
        \cellcolor{black} & $x_r$ & \texttt{<} \\
        \hline
	\cellcolor{black}  & s & \cellcolor{black} \\
        \hline
        \cellcolor{black} & \cellcolor{black} & \cellcolor{black} \\
        \hline
        \end{tabular}
    };

    \node (tab12) at (9, -4){
        \begin{tabular}{| >{\centering\arraybackslash}m{0.75cm}| >{\centering\arraybackslash}m{0.75cm}| >{\centering\arraybackslash}m{0.75cm}|}
        \hline
        \texttt{>} \\
        \hline
	\cellcolor{black}  \\
        \hline
          \cellcolor{black} \\
        \hline
        \end{tabular}
    };

    \draw[->, thick] (tab11) -- (tab12);
    \node at (7.5, -3.2) {$s \neq (q_h, b)$};
    
    \node (tab13) at (0,-6) {
        \begin{tabular}{| >{\centering\arraybackslash}m{0.75cm}| >{\centering\arraybackslash}m{0.75cm}| >{\centering\arraybackslash}m{0.75cm}|}
        \hline
        \cellcolor{black} & $x_{r/l/0}$ & \cellcolor{black}\\
        \hline
	\cellcolor{black} & $q_h, b$ & \cellcolor{black} \\
        \hline
        \cellcolor{black} & \cellcolor{black} & \cellcolor{black} \\
        \hline
        \end{tabular}
    };

    \node (tab14) at (3,-6) {
        \begin{tabular}{| >{\centering\arraybackslash}m{0.75cm}| >{\centering\arraybackslash}m{0.75cm}| >{\centering\arraybackslash}m{0.75cm}|}
        \hline
	 \\
        \hline
        \cellcolor{black} \\
        \hline
          \cellcolor{black} \\
        \hline
        \end{tabular}
    };
    \draw[->, thick] (tab13) -- (tab14);
    
    \node (tab15) at (6,-6) {
        \begin{tabular}{| >{\centering\arraybackslash}m{0.75cm}| >{\centering\arraybackslash}m{0.75cm}| >{\centering\arraybackslash}m{0.75cm}|}
        \hline
        $x_r$ &  & \cellcolor{black} \\
        \hline
	s & \cellcolor{black} & \cellcolor{black}\\
        \hline
        \cellcolor{black} & \cellcolor{black} & \cellcolor{black} \\
        \hline
        \end{tabular}
    };

    \node (tab16) at (9, -6){
        \begin{tabular}{| >{\centering\arraybackslash}m{0.75cm}| >{\centering\arraybackslash}m{0.75cm}| >{\centering\arraybackslash}m{0.75cm}|}
        \hline
        $x_0$ \\
        \hline
	\cellcolor{black} \\
        \hline
          \cellcolor{black} \\
        \hline
        \end{tabular}
    };

    \draw[->, thick] (tab15) -- (tab16);
    \node at (7.5, -5.2) {$s \neq (q_h, b)$};

    \node (tab17) at (0,-8) {
        \begin{tabular}{| >{\centering\arraybackslash}m{0.75cm}| >{\centering\arraybackslash}m{0.75cm}| >{\centering\arraybackslash}m{0.75cm}|}
        \hline
        \cellcolor{black} & & \cellcolor{black}\\
        \hline
	\cellcolor{black} & a & \cellcolor{black} \\
        \hline
        \cellcolor{black} & a' & \cellcolor{black}\\
        \hline
        \end{tabular}
    };

    \node (tab18) at (3,-8) {
        \begin{tabular}{| >{\centering\arraybackslash}m{0.75cm}| >{\centering\arraybackslash}m{0.75cm}| >{\centering\arraybackslash}m{0.75cm}|}
        \hline
	\cellcolor{black} \\
        \hline
         \\
        \hline
        \\
        \hline
        \end{tabular}
    };
    \draw[->, thick] (tab17) -- (tab18);
    
    \node (tab19) at (6,-8) {
        \begin{tabular}{| >{\centering\arraybackslash}m{0.75cm}| >{\centering\arraybackslash}m{0.75cm}| >{\centering\arraybackslash}m{0.75cm}|}
        \hline
        \texttt{>} & $x_0$ & \cellcolor{black}\\
        \hline
	\cellcolor{black} & a & \cellcolor{black} \\
        \hline
        \cellcolor{black} & a' & \cellcolor{black}\\
        \hline
        \end{tabular}
    };

    \node (tab20) at (9, -8){
        \begin{tabular}{| >{\centering\arraybackslash}m{0.75cm}| >{\centering\arraybackslash}m{0.75cm}| >{\centering\arraybackslash}m{0.75cm}|}
        \hline
        \cellcolor{black} \\
        \hline
	a' \\
        \hline
        a'\\
        \hline
        \end{tabular}
    };

    \draw[->, thick] (tab19) -- (tab20);
    
    \node (tab21) at (0,-10) {
        \begin{tabular}{| >{\centering\arraybackslash}m{0.75cm}| >{\centering\arraybackslash}m{0.75cm}| >{\centering\arraybackslash}m{0.75cm}|}
        \hline
        \cellcolor{black} & \texttt{>}  & \\
       	\hline
	\cellcolor{black} & \cellcolor{black} & \cellcolor{black} \\
        \hline
        \cellcolor{black} & \cellcolor{black} & \cellcolor{black} \\
        \hline
        \end{tabular}
    };

    \node (tab22) at (3,-10) {
        \begin{tabular}{| >{\centering\arraybackslash}m{0.75cm}| >{\centering\arraybackslash}m{0.75cm}| >{\centering\arraybackslash}m{0.75cm}|}
        \hline
          \\
       	\hline
	  \cellcolor{black} \\
        \hline
          \cellcolor{black} \\
        \hline
        \end{tabular}
    };
    \draw[->, thick] (tab21) -- (tab22);
    
    \node (tab23) at (6,-10) {
        \begin{tabular}{| >{\centering\arraybackslash}m{0.75cm}| >{\centering\arraybackslash}m{0.75cm}| >{\centering\arraybackslash}m{0.75cm}|}
        \hline
         & \texttt{<} & \cellcolor{black} \\
        \hline
        \cellcolor{black} & \cellcolor{black} & \cellcolor{black} \\
        \hline
        \cellcolor{black} & \cellcolor{black} & \cellcolor{black} \\
        \hline
        \end{tabular}
    };

    \node (tab24) at (9, -10){
        \begin{tabular}{| >{\centering\arraybackslash}m{0.75cm}| >{\centering\arraybackslash}m{0.75cm}| >{\centering\arraybackslash}m{0.75cm}|}
        \hline
          \\
        \hline
          \cellcolor{black} \\
        \hline
          \cellcolor{black} \\
        \hline
        \end{tabular}
    };

    \draw[->, thick] (tab23) -- (tab24);
    
    \node (tab25) at (0,-12) {
        \begin{tabular}{| >{\centering\arraybackslash}m{0.75cm}| >{\centering\arraybackslash}m{0.75cm}| >{\centering\arraybackslash}m{0.75cm}|}
        \hline
        \cellcolor{black} & p & \cellcolor{black}\\
        \hline
        \cellcolor{black} & \cellcolor{black} & \cellcolor{black} \\
        \hline
        \cellcolor{black} & \cellcolor{black} & \cellcolor{black} \\
        \hline
        \end{tabular}
    };

    \node (tab26) at (3,-12) {
        \begin{tabular}{| >{\centering\arraybackslash}m{0.75cm}| >{\centering\arraybackslash}m{0.75cm}| >{\centering\arraybackslash}m{0.75cm}|}
        \hline
	\\
        \hline
          \cellcolor{black} \\
        \hline
          \cellcolor{black} \\
        \hline
        \end{tabular}
    };
    \draw[->, thick] (tab25) -- (tab26);
    
    \node (tab27) at (6,-12) {
        \begin{tabular}{| >{\centering\arraybackslash}m{0.75cm}| >{\centering\arraybackslash}m{0.75cm}| >{\centering\arraybackslash}m{0.75cm}|}
        \hline
         p &  &\cellcolor{black} \\
        \hline 
        \cellcolor{black} & \cellcolor{black} & \cellcolor{black} \\
        \hline
        \cellcolor{black} & \cellcolor{black} & \cellcolor{black} \\
        \hline
        \end{tabular}
    };

    \node (tab28) at (9, -12){
        \begin{tabular}{| >{\centering\arraybackslash}m{0.75cm}| >{\centering\arraybackslash}m{0.75cm}| >{\centering\arraybackslash}m{0.75cm}|}
        \hline
        p \\
        \hline
          \cellcolor{black} \\
        \hline
          \cellcolor{black} \\
        \hline
        \end{tabular}
    };

    \draw[->, thick] (tab27) -- (tab28);
    
    \node (tab29) at (0,-14) {
        \begin{tabular}{| >{\centering\arraybackslash}m{0.75cm}| >{\centering\arraybackslash}m{0.75cm}| >{\centering\arraybackslash}m{0.75cm}|}
        \hline
        \cellcolor{black} & \texttt{>} & \texttt{<} \\
        \hline
        \cellcolor{black} & \cellcolor{black} & \cellcolor{black} \\
        \hline
        \cellcolor{black} & \cellcolor{black} & \cellcolor{black} \\
        \hline
        \end{tabular}
    };

    \node (tab30) at (3,-14) {
        \begin{tabular}{| >{\centering\arraybackslash}m{0.75cm}| >{\centering\arraybackslash}m{0.75cm}| >{\centering\arraybackslash}m{0.75cm}|}
        \hline
	$x_r$\\
        \hline
          \cellcolor{black} \\
        \hline
          \cellcolor{black} \\
        \hline
        \end{tabular}
    };
    \draw[->, thick] (tab29) -- (tab30);

    \node (tab31) at (6,-14) {
        \begin{tabular}{| >{\centering\arraybackslash}m{0.75cm}| >{\centering\arraybackslash}m{0.75cm}| >{\centering\arraybackslash}m{0.75cm}|}
        \hline
	\cellcolor{black} & \texttt{<} & \cellcolor{black} \\
        \hline
        \cellcolor{black} & \cellcolor{black} & \cellcolor{black} \\
        \hline
         & a & \cellcolor{black}\\
         \hline
        \end{tabular}
    };

    \node (tab32) at (9, -14){
        \begin{tabular}{| >{\centering\arraybackslash}m{0.75cm}| >{\centering\arraybackslash}m{0.75cm}| >{\centering\arraybackslash}m{0.75cm}|}
        \hline
        \cellcolor{black}\\
        \hline
	\cellcolor{black} \\
        \hline
        \\
        \hline
        \end{tabular}
    };

    \draw[->, thick] (tab31) -- (tab32);

    \node (tab33) at (0,-16) {
        \begin{tabular}{| >{\centering\arraybackslash}m{0.75cm}| >{\centering\arraybackslash}m{0.75cm}| >{\centering\arraybackslash}m{0.75cm}|}
        \hline
	\texttt{<} & $x_l$ & \cellcolor{black} \\
        \hline
        \cellcolor{black} & s & \cellcolor{black} \\
        \hline
        \cellcolor{black} & \cellcolor{black} & \cellcolor{black} \\
        \hline
        \end{tabular}
    };

    \node (tab34) at (3,-16) {
        \begin{tabular}{| >{\centering\arraybackslash}m{0.75cm}| >{\centering\arraybackslash}m{0.75cm}| >{\centering\arraybackslash}m{0.75cm}|}
        \hline
	$x_r$\\
        \hline
        \cellcolor{black} \\
        \hline
          \cellcolor{black} \\
        \hline
        \end{tabular}
    };
    \draw[->, thick] (tab33) -- (tab34);
     \node at (1.5, -15.1) {$s \neq (q_h, b)$};
    
    \node (tab35) at (6,-16) {
        \begin{tabular}{| >{\centering\arraybackslash}m{0.75cm}| >{\centering\arraybackslash}m{0.75cm}| >{\centering\arraybackslash}m{0.75cm}|}
        \hline
        \cellcolor{black} & $x_r$ & \texttt{>} \\
        \hline
	\cellcolor{black} & s & \cellcolor{black} \\
        \hline
        \cellcolor{black} & \cellcolor{black} & \cellcolor{black} \\
        \hline
        \end{tabular}
    };

    \node (tab36) at (9, -16){
        \begin{tabular}{| >{\centering\arraybackslash}m{0.75cm}| >{\centering\arraybackslash}m{0.75cm}| >{\centering\arraybackslash}m{0.75cm}|}
        \hline
         $x_l$ \\
        \hline
        \cellcolor{black}\\
        \hline
          \cellcolor{black} \\
        \hline
        \end{tabular}
    };

    \draw[->, thick] (tab35) -- (tab36);
     \node at (7.5, -15.1) {$s \neq (q_h, b)$};

    \node (tab37) at (0,-18) {
        \begin{tabular}{| >{\centering\arraybackslash}m{0.75cm}| >{\centering\arraybackslash}m{0.75cm}| >{\centering\arraybackslash}m{0.75cm}|}
        \hline
        \texttt{<} & $x_0$ & \cellcolor{black} \\
        \hline
	\cellcolor{black} & a & \cellcolor{black} \\
        \hline
        \cellcolor{black} & a' & \cellcolor{black}\\
        \hline
        \end{tabular}
    };

    \node (tab38) at (3, -18){
        \begin{tabular}{| >{\centering\arraybackslash}m{0.75cm}| >{\centering\arraybackslash}m{0.75cm}| >{\centering\arraybackslash}m{0.75cm}|}
        \hline
         \cellcolor{black} \\
        \hline
	$q_i$, a'  \\
        \hline
        a'\\
        \hline
        \end{tabular}
    };

    \draw[->, thick] (tab37) -- (tab38);

    \node (tab39) at (6,-18) {
        \begin{tabular}{| >{\centering\arraybackslash}m{0.75cm}| >{\centering\arraybackslash}m{0.75cm}| >{\centering\arraybackslash}m{0.75cm}|}
        \hline
	\cellcolor{black} & $x_r$ & \\
        \hline
        \cellcolor{black} & s & \cellcolor{black}\\
        \hline
        \cellcolor{black} & \cellcolor{black} & \cellcolor{black} \\
        \hline
        \end{tabular}
    };

    \node (tab40) at (9, -18){
        \begin{tabular}{| >{\centering\arraybackslash}m{0.75cm}| >{\centering\arraybackslash}m{0.75cm}| >{\centering\arraybackslash}m{0.75cm}|}
        \hline
        \texttt{>} \\
        \hline
	\cellcolor{black} \\
        \hline
          \cellcolor{black} \\
        \hline
        \end{tabular}
    };

    \draw[->, thick] (tab39) -- (tab40);
    \node at (7.5, -17.2) {$s \neq (q_h, b)$};

\end{tikzpicture}
\end{center}
\caption{Complete transition function of the cellular automaton $G_e$. Black cells act as wildcards and can match any state in the neighborhood. When no rule applies, the cell remains unchanged.}
\label{fig:transition1D}
\end{figure}

\subsection{Lower Bound and $\Pi^0_2$-Hardness}

\begin{lemma}\label{lemma:1Dhard}
For any $e \in \mathbb{N}$, the Turing machine $M_e$ halts on all inputs if and only if the cellular automaton $G_e$ is sensitive.
\end{lemma}

\begin{proof}
We now prove that $e \in \text{TOT}$ if and only if $G_e$ is sensitive.

\paragraph{If $e \in \text{TOT}$:} We need to show that $G_e$ is sensitive. Let $w$ be any word. We can take $y \in \text{Cyl}(w)$ such that $y$ has only blank symbols on all tapes to the right of $w$, and to the left, there is only one $p$ particle moving to the right and blank symbols.

Since $e \in \text{TOT}$, all computational blocks in $w$ will eventually be destroyed. Indeed, in the rightmost computational block in $w$, the $x$ symbols will eventually be inserted and then be able to extend the computational block arbitrarily far, allowing the Turing machine to compute on its input without being restarted and with arbitrary large space. Hence, the rightmost Turing machine will halt and erase itself in finitely many steps, leaving space for the machine to its left. Therefore, in finitely many steps, $w$ will become blank. Thus, the particle we placed to the left of $w$ can pass through $w$ if the particle is placed far enough. Therefore, since we can choose whether or not to place the particle, this establishes sensitivity.

\paragraph{If $e \notin \text{TOT}$:} Then there exists an input $n$ such that $M_e$ runs forever on $n$ without halting. Consider the word pattern $uu$ where $u$ is a computational block of size at least radius of $G_e$ containing input $n$. The machine simulated in the second copy of $u$ will never halt and will never be destroyed. Consequently, the machine in the first copy of $u$ remains confined to its computational block and likewise never halts. This creates a barrier that no information can cross. Therefore $uu$ is an $r$-blocking word, proving that $G_e$ is not sensitive. \qed
\end{proof}

Thus, we have shown that $G_e$ is sensitive if and only if $e \in \text{TOT}$.

\begin{proof}[Theorem~\ref{thm:dim1}]
By Lemma~\ref{lemma:sensitivity-pi02}, we know that the sensitivity problem is in $\Pi^0_2$, and by Lemma~\ref{lemma:1Dhard}, we have hardness. We conclude that determining sensitivity for one-dimensional cellular automata is $\Pi^0_2$-complete. This proves the theorem. \qed
\end{proof}

As a corollary of these constructions, we obtain a new proof of Sutner's result on finite nilpotency \cite{sutner1989}.
\begin{corollary}
The problem of determining whether a cellular automaton is nilpotent on finite configurations is $\Pi^0_2$-complete. \end{corollary}

\begin{proof}
Given a one-dimensional cellular automaton $G$ of radius $r$, finite nilpotency is a $\Pi^0_2$ property. This is because finite configurations can be encoded by a finite set along with the position of the top-rightmost non-zero cell. Consequently, the function $S$ that takes the code of a finite configuration $c$ and returns the code of $G(c)$ is recursive. Therefore finite nilpotency can be written as :
\begin{equation}
\forall e\in \mathbb N, \quad \exists n\in \mathbb N, \quad S^n(e) = \text{Code(0)}
\end{equation}
We obtain hardness by reducing TOT and removing particle $p$ in our construction of $G_e$. Hence, from our previous proof, all finite patterns $P$ will eventually become blank if and only if the given Turing machine halts on every input. We can then lift the result to higher dimensions using slices. For any $d$ and a $d$-dimensional cellular automaton $G$ which is finitely nilpotent. Consider the slice version of $G$. Any finite configuration on $A^{\mathbb{Z}^{d+1}}$ is finite on every slice and non-null on finitely many. Therefore, the slice version of $G$ is finitely nilpotent if and only if $G$ is. By induction, finite nilpotency is $\Pi^0_2$-complete. \qed
\end{proof}

\section{Higher Dimensional Cellular Automata}\label{sec:proof2}

In this section, we extend our analysis to cellular automata in two or more dimensions and prove Theorem~\ref{thm:dimd}.

\subsection{Upper Bound: $\Sigma^0_3$}

We begin by showing that for any $d > 1$, the sensitivity problem for $d$-dimensional cellular automata belongs to the class $\Sigma^0_3$ of the arithmetical hierarchy. First, we recall a crucial lemma from our previous discussion, Lemma~\ref{lemma:blocking-words}: A cellular automaton is not sensitive if and only if it has an arbitrarily large blocking word. We then state the following lemma.

\begin{lemma}\label{lemma:uBound}
For any $d > 1$, the problem of determining whether a $d$-dimensional cellular automaton is sensitive is $\Sigma^0_3$.
\end{lemma}

\begin{proof}
Recall that the property "there exists an $M$-blocking word" is $\Sigma^0_2$ by Lemma~\ref{lemma:blocking-word}. Therefore, by Lemma~\ref{lemma:blocking-words}, being non-sensitive is $\Pi^0_3$, as it requires the existence of $M$-blocking words for all sufficiently large $M$. Consequently, being sensitive is $\Sigma^0_3$. This establishes that the sensitivity problem is at most $\Sigma^0_3$. \qed
\end{proof}

\subsection{Lower Bound: $\Sigma^0_3$-Hardness}

In this section, we aim to prove the following result:

\begin{theorem}\label{thm:2D}
The problem of determining whether a two-dimensional cellular automaton is sensitive is $\Sigma^0_3$-complete.
\end{theorem}

To prove $\Sigma^0_3$-hardness, we reduce from the Cofinite Set (COF) problem, which is known to be $\Sigma^0_3$-hard. For each $e$, let $M_e = (Q, A, q_i, q_h, \delta)$ be the associated Turing machine. We construct a 2D cellular automaton $G_e$ as follows:

The states of $G_e$ are $(A \cup \{\sqcup\}) \times (Q \cup \{\sqcup\})  \times (T \cup \{\sqcup\}) \times (P \cup \{\sqcup\}) \times \{\text{red1}, \text{red2}, \text{red3}, \text{red4}, \text{white}, \text{green}\}$.

Where $P= \{p, p_r, p_l, q_r, q_l, p_{ru}, p_{rd}, p_{lu}, p_{ld}, q_{ru}, q_{rd}, q_{lu}, q_{ld}\}$ is a set of particles, $T = Q\cup \{t_h\}$, with $Q$ a set of states for tentacles and $t_h$ a signal for halting. The colors red, white, and green delimit the computational zones. 

 $G_e$ operates using the von Neumann neighborhood.

We describe the cellular automaton's behavior as follows:

Each red cell can have at most two red neighbors. A red cell marked with $1$ can have either a $1$ neighbor to the west and a $1$ neighbor to the east, a $2$ neighbor to the west and a $1$ neighbor to the east, or a $4$ neighbor to the north and a $1$ neighbor to the west. Other marker rules are given Figure~\ref{fig:transition2D}. Consequently given a finite configuration, after finitely many steps, the only remaining red zones are the ones forming rectangular loops (Lemma~\ref{lem:redzone}),  these rectangles will form the tape of Turing machines. 

    
 $P$ is a set of particles, that move on white or green cells. There are two types of particles: $p$ particles, which send a signal to a Turing machine to process one step of computation, and $q$ particles, which place a Turing machine at the beginning of the tape in the initial state. The particles $p$ and $q$ are divided into right and left particles ($p_r$, $p_l$) that move to the right and left, respectively. When they meet at the beginning of a tape, they send a signal into that tape. If a particle meets an obstacle (i.e. a red zone) it divides into an up particle and a down particle ($p_{ru}$, $p_{lu}$, $p_{rd}$, $p_{ld}$) that follow the red loop and reform when meeting again. 
 
When a $p$-signal enters a red zone, it follows the loop guided by the 1, 2, 3, and 4 markers. When it encounters a Turing machine head, it allows the head to process one step of computation and is erased after meeting the first head or if there is no head on the zone. The Turing machine head is guided along the loop by the markers. If it encounters another head on the tape, it erases that head (i.e. the most advanced particle on the zone is erased). When it halts, it erases itself and the red zone.
 
When a machine runs out of space, it requests the creation of a tentacle that operates with $T$ symbols on green zones. Tentacles always grow to the east, or to the north if there is a red zone in the eastward position. These extension rules define local rules, and tentacles can be defined as green zones respecting these rules. The set $T$ contains copies of states of $M_e$ and includes a signal $t_h$ that is sent when the machine halts to destroy the connected red zone. By the first rules, when two green zones meet, they erase each other.If a particle reaches the end of the red zone (meaning it has not found any Turing machine head), it continues into the green zone.


We give below the main rules of the cellular automaton $G_e$ in Figure~\ref{fig:transition2D}.

\input{transition2D.tex}

\begin{lemma}\label{lem:redzone}
Given a finite configuration, after finitely many steps all red zones will form rectangular loops, and all green zones will become tentalces. This process is illustrated in Figure~\ref{fig:redzone}.
\end{lemma}
\begin{proof}
Given a finite configuration, each green cell that does not follow the local rules will eventually become white. Therefore, after finitely many steps, all green zones that violate local rules will be erased. The only remaining connected green zones will be those that comply with the local rules, which take the form of tentacles.

Now consider the 5-state cellular automaton with states \(\{0, 1, 2, 3, 4\}\) governed by the transition rules \texttt{red1}, \texttt{red2}, \texttt{red3}, \texttt{red4}, and \texttt{white}, as shown in Figure~\ref{fig:transition2D}. Given a 0-finite configuration, the local rules constrain any connected non-zero zone to form a spiral (the proof is omitted). There are two possibilities: either the spiral closes into a rectangular loop, or it spirals indefinitely. In the latter case, the inward-spiraling zone creates a conflict at its center. Because no non-zero cell can be created the whole non-zero zone will be erased.  \qed
\end{proof}

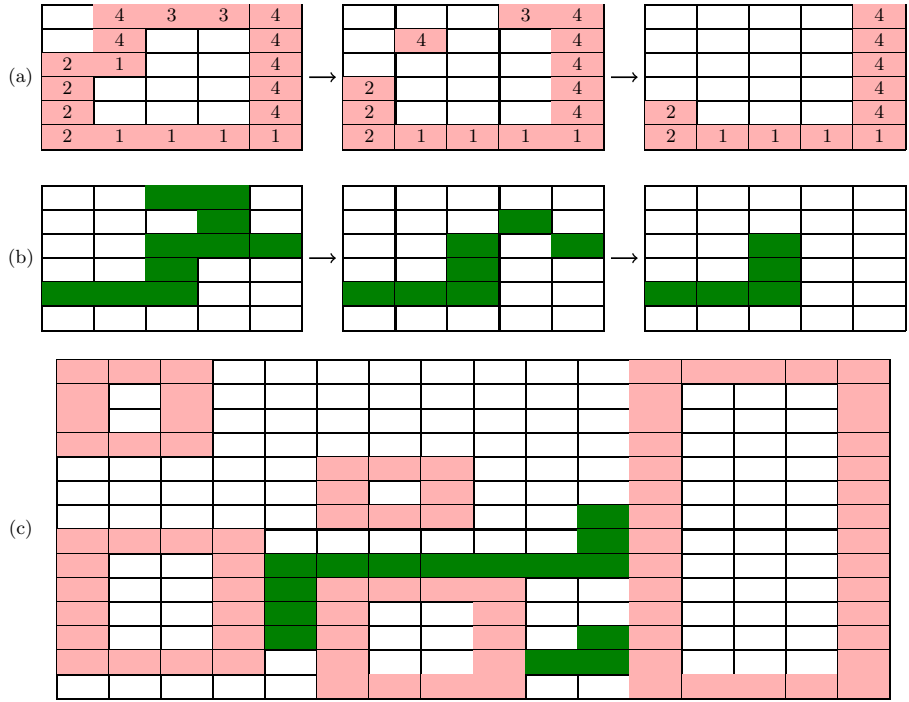
\begin{figure}[h!]
	\begin{center}
	\resizebox{\linewidth}{!}{
\begin{tikzpicture}

    \node (tab1) at (0,0) {
        \begin{tabular}{| >{\centering\arraybackslash}m{0.75cm}
                  | >{\centering\arraybackslash}m{0.75cm}
                  | >{\centering\arraybackslash}m{0.75cm}
                  | >{\centering\arraybackslash}m{0.75cm}
                  | >{\centering\arraybackslash}m{0.75cm} |}
        \hline
         &  \cellcolor{red!30}4 &\cellcolor{red!30}3 &\cellcolor{red!30}3 &\cellcolor{red!30}4 \\
        \hline
	 & \cellcolor{red!30}4 & & &\cellcolor{red!30}4\\
        \hline
        \cellcolor{red!30}2 & \cellcolor{red!30}1  & & &\cellcolor{red!30}4\\
        \hline
        \cellcolor{red!30}2&  & & &\cellcolor{red!30}4\\
        \hline
        \cellcolor{red!30}2&  & & & \cellcolor{red!30}4\\
        \hline
        \cellcolor{red!30}2 &  \cellcolor{red!30}1 & \cellcolor{red!30}1& \cellcolor{red!30}1 & \cellcolor{red!30}1\\
        \hline
        \end{tabular}
    };

    \node (tab2) at (5, 0){
        \begin{tabular}{| >{\centering\arraybackslash}m{0.75cm}
                  | >{\centering\arraybackslash}m{0.75cm}
                  | >{\centering\arraybackslash}m{0.75cm}
                  | >{\centering\arraybackslash}m{0.75cm}
                  | >{\centering\arraybackslash}m{0.75cm} |}
        \hline
         &   & &\cellcolor{red!30}3 &\cellcolor{red!30}4 \\
        \hline
	 & \cellcolor{red!30}4 & & &\cellcolor{red!30}4\\
        \hline
         &  & & &\cellcolor{red!30}4\\
        \hline
        \cellcolor{red!30}2&  & & &\cellcolor{red!30}4\\
        \hline
        \cellcolor{red!30}2&  & & & \cellcolor{red!30}4\\
        \hline
        \cellcolor{red!30}2 &  \cellcolor{red!30}1 & \cellcolor{red!30}1& \cellcolor{red!30}1 & \cellcolor{red!30}1\\
        \hline
        \end{tabular}
    };
    \node (tab3) at (10, 0){
        \begin{tabular}{| >{\centering\arraybackslash}m{0.75cm}
                  | >{\centering\arraybackslash}m{0.75cm}
                  | >{\centering\arraybackslash}m{0.75cm}
                  | >{\centering\arraybackslash}m{0.75cm}
                  | >{\centering\arraybackslash}m{0.75cm} |}
        \hline
         &   & & &\cellcolor{red!30}4 \\
        \hline
	 &  & & &\cellcolor{red!30}4\\
        \hline
         &   & & &\cellcolor{red!30}4\\
        \hline
        &  & & &\cellcolor{red!30}4\\
        \hline
        \cellcolor{red!30}2&  & & & \cellcolor{red!30}4\\
        \hline
        \cellcolor{red!30}2 &  \cellcolor{red!30}1 & \cellcolor{red!30}1& \cellcolor{red!30}1 & \cellcolor{red!30}1\\
        \hline
        \end{tabular}
    };

    \draw[->, thick] (tab1) -- (tab2);
    \draw[->, thick] (tab2) -- (tab3);

    \node (tab5) at (0,-3) {
         \begin{tabular}{|p{0.75cm}|p{0.75cm}|p{0.75cm}|p{0.75cm}|p{0.75cm}|}
        \hline
         &  & \cellcolor{green!50!black} & \cellcolor{green!50!black} & \\
        \hline
	 & & & \cellcolor{green!50!black} &\\
        \hline
         &  & \cellcolor{green!50!black}& \cellcolor{green!50!black} & \cellcolor{green!50!black}\\
        \hline
        &  & \cellcolor{green!50!black}& &\\
        \hline
        \cellcolor{green!50!black} & \cellcolor{green!50!black}& \cellcolor{green!50!black}& &\\
        \hline
        &  & & &\\
        \hline
        \end{tabular}    
     };

    \node (tab6) at (5,-3) {
         \begin{tabular}{|p{0.75cm}|p{0.75cm}|p{0.75cm}|p{0.75cm}|p{0.75cm}|}
        \hline
         &  &  & & \\
        \hline
	 & & & \cellcolor{green!50!black} &\\
        \hline
         &  & \cellcolor{green!50!black}& & \cellcolor{green!50!black}\\
        \hline
        &  & \cellcolor{green!50!black}& &\\
        \hline
        \cellcolor{green!50!black} & \cellcolor{green!50!black}& \cellcolor{green!50!black}& &\\
        \hline
        &  & & &\\
        \hline
        \end{tabular}       
      };
    
     \node (tab7) at (10,-3) {
         \begin{tabular}{|p{0.75cm}|p{0.75cm}|p{0.75cm}|p{0.75cm}|p{0.75cm}|}
        \hline
         &  &  & & \\
        \hline
	 & & & &\\
        \hline
         &  & \cellcolor{green!50!black}& & \\
        \hline
        &  & \cellcolor{green!50!black}& &\\
        \hline
        \cellcolor{green!50!black} & \cellcolor{green!50!black}& \cellcolor{green!50!black}& &\\
        \hline
        &  & & &\\
        \hline
        \end{tabular}           
     };

    \draw[->, thick] (tab5) -- (tab6);
     \draw[->, thick] (tab6) -- (tab7);
    
    \node (tab8) at (5,-7.5) {
         \begin{tabular}{|p{0.75cm}|p{0.75cm}|p{0.75cm}|p{0.75cm}|p{0.75cm}|p{0.75cm}|p{0.75cm}|p{0.75cm}|p{0.75cm}|p{0.75cm}|p{0.75cm}|p{0.75cm}|p{0.75cm}|p{0.75cm}|p{0.75cm}|p{0.75cm}|p{0.75cm}|}
        \hline
         \cellcolor{red!30}& \cellcolor{red!30} &\cellcolor{red!30} & & & & & & & & & \cellcolor{red!30}&\cellcolor{red!30} &\cellcolor{red!30} &\cellcolor{red!30} &\cellcolor{red!30} \\
        \hline
	\cellcolor{red!30}&  & \cellcolor{red!30}& & & & & & & & &\cellcolor{red!30} & & & &\cellcolor{red!30} \\
        \hline
        \cellcolor{red!30}& & \cellcolor{red!30}& & & & & & & & &\cellcolor{red!30} & & & &\cellcolor{red!30} \\
        \hline
        \cellcolor{red!30}&\cellcolor{red!30}  &\cellcolor{red!30} & & & & & & & & &\cellcolor{red!30} & & & &\cellcolor{red!30} \\
        \hline
        &  & & & & \cellcolor{red!30}& \cellcolor{red!30} & \cellcolor{red!30}& & & &\cellcolor{red!30} & & & &\cellcolor{red!30} \\
        \hline
        &  & & & & \cellcolor{red!30}& &\cellcolor{red!30} & & & & \cellcolor{red!30}& & & &\cellcolor{red!30} \\
        \hline
        &  & & & &\cellcolor{red!30} &\cellcolor{red!30} & \cellcolor{red!30}& & & \cellcolor{green!50!black}&\cellcolor{red!30} & & & &\cellcolor{red!30} \\
        \hline
        \cellcolor{red!30}& \cellcolor{red!30} &\cellcolor{red!30} & \cellcolor{red!30}& & & & & & &\cellcolor{green!50!black} &\cellcolor{red!30} & & & &\cellcolor{red!30} \\
        \hline
        \cellcolor{red!30}&  & & \cellcolor{red!30}&\cellcolor{green!50!black} &\cellcolor{green!50!black} &\cellcolor{green!50!black} & \cellcolor{green!50!black}& \cellcolor{green!50!black}&\cellcolor{green!50!black} & \cellcolor{green!50!black} & \cellcolor{red!30}& & & & \cellcolor{red!30}\\
        \hline
        \cellcolor{red!30}&  & & \cellcolor{red!30}&\cellcolor{green!50!black} &\cellcolor{red!30} &\cellcolor{red!30} &\cellcolor{red!30} & \cellcolor{red!30}& & &\cellcolor{red!30} & & & &\cellcolor{red!30} \\
        \hline
        \cellcolor{red!30}&  & &\cellcolor{red!30} & \cellcolor{green!50!black}&\cellcolor{red!30} & & &\cellcolor{red!30} & & & \cellcolor{red!30}& & & & \cellcolor{red!30}\\
        \hline
        \cellcolor{red!30}&  & & \cellcolor{red!30}& \cellcolor{green!50!black}&\cellcolor{red!30} & & &\cellcolor{red!30} & &\cellcolor{green!50!black} & \cellcolor{red!30}& & & &\cellcolor{red!30} \\
        \hline
        \cellcolor{red!30}&\cellcolor{red!30}  &\cellcolor{red!30} &\cellcolor{red!30} & & \cellcolor{red!30}& & &\cellcolor{red!30} &\cellcolor{green!50!black} &\cellcolor{green!50!black} & \cellcolor{red!30}& & & &\cellcolor{red!30} \\
        \hline
        &  & & & & \cellcolor{red!30}&\cellcolor{red!30} & \cellcolor{red!30}& \cellcolor{red!30}& & &\cellcolor{red!30} &\cellcolor{red!30} &\cellcolor{red!30} &\cellcolor{red!30} &\cellcolor{red!30} \\
        \hline
        \end{tabular}    
     };

\node at (-2.5,0) {(a)};
\node at (-2.5,-3) {(b)};
\node at (-2.5,-7.5) {(c)};

\end{tikzpicture}}
\end{center}
\caption{Illustration of the stabilization process. (a) Defective red regions are progressively eliminated. (b) Green regions are removed in parallel. (c) A representative configuration after the system has stabilized.}
\label{fig:redzone}
\end{figure}

\begin{lemma}
	Given a finite configuration, we can send particles from outside the configuration to meet at any position outside the red zones within the configuration after the red zones have stabilized.
\end{lemma}

\begin{proof}
		Consider a finite configuration in which the only red zones are rectangular loops, and suppose there is a single right-moving particle (possibly split) to the left of a left-moving particle (possibly split), both located outside the red zones. This configuration has a preimage. This follows from the fact that red zones are rectangular, allowing us to trace the evolution backward using the local rules. Moreover, since the right-moving particle is to the left of the left-moving one, they could not have met in the past. 
We place the particles at the desired meeting point and trace the consecutive preimages of the configuration backward. This allows us to determine where to initially place the particles outside the configuration so that they meet as intended. We only need to place them sufficiently far from the configuration to ensure that the red zones have time to evolve into rectangular loops. \qed
\end{proof}

\begin{lemma}\label{lemma:2Dhard}
For any $e \in \mathbb{N}$, the Turing machine $M_e$ halts on all but finitely many inputs if and only if the associated cellular automaton $G_e$ is sensitive.   
\end{lemma}

\begin{proof}
We prove that $e \in \text{COF}$ if and only if $G_e$ is sensitive. Let $x$ be any finite configuration, by Lemma~\ref{lem:redzone} in finite time all red zone will be rectangular red loops. These rectangular loops will be our Turing machines tapes and the input of a Turing machines will the length of its loop/tape divided by two. 

\paragraph{If $e \in \text{COF}$:} Let $x$ be any finite configuration. In finite time it will be constituted of rectangular red loops, white zones, and tentacles. For any accessible red zone (not inside another), we can send a particles $q$ to place a Turing machine on the zone. We can then send processing particles in those zones. Because inaccessible zones are surrounded by red zones, if these zones are sufficiently large, they must be surrounded by a large red area. Since $e \in \text{COF}$, there exists an $N$ such that for all $y > N$, $M_e$ halts on $y$. These machines might not have enough space to complete their computation, but they can request a tentacle. Since we control which zone is allowed to compute by sending a particle, we can prevent conflicts in the green zone by ensuring that only one red zone is permitted to compute at a time. Therefore, any sufficiently large block will have its border destroyed because the Turing on that tape will eventually halt, allowing us to send particles through it. by Lemma~\ref{lemma:blocking-words}, $G_e$ is sensitive.

\paragraph{If $e \notin \text{COF}$:} Then there exist infinitely many $y$ such that $M_e(y)$ does not halt. We can construct arbitrarily large blocking words as follows: For any $M > 0$, choose $y > M$ such that $M_e(y)$ does not halt. Construct a red rectangle with perimeter $2y$. This rectangle forms an $M$-blocking word, as:

1. The Turing machine simulating $M_e(y)$ will never halt, so the rectangle will never be destroyed.
2. No particles can penetrate the red border.
3. The interior of the rectangle is inaccessible to any external influence.

Since we can construct such blocking words for arbitrarily large $M$, by Lemma \ref{lemma:blocking-words}, $G_e$ is not sensitive. \qed
\end{proof}

\begin{proof}[Theorem~\ref{thm:2D}]
By Lemma~\ref{lemma:uBound}, we know that the sensitivity problem for two-dimensional cellular automata is $\Sigma^0_3$, and by Lemma~\ref{lemma:2Dhard}, we have established hardness. Hence, the theorem is proved. \qed
\end{proof}

\begin{proof}[Theorem~\ref{thm:dimd}]
By Lemma~\ref{lemma:uBound}, we know that the sensitivity problem for $d$-dimensional cellular automata is in $\Sigma^0_3$. We prove $\Sigma^0_3$-hardness by reducing from the two-dimensional case using a slicing argument.

Let $G_d$ be a $d$-dimensional cellular automaton, and let $G_{d+1}$ be its slice extension to dimension $d+1$. Suppose $w_d$ is an $m$-blocking word for $G_d$ over some finite set $K \subset \mathbb{Z}^d$. Then the word
\begin{equation}
w_{d+1}(x, x_{d+1}) = w_d(x), \quad \forall x_{d+1} \in [0, m-1], \quad \forall x \in K
\end{equation}
defines an $m$-blocking word for $G_{d+1}$. Therefore, if $G_d$ is not sensitive, neither is $G_{d+1}$.

Conversely, if $G_d$ is sensitive, then it does not admit arbitrarily large blocking words, so information can propagate in slices. Hence, $G_{d+1}$ is also sensitive. By induction on $d$, this proves the theorem. \qed
\end{proof}

This construction establishes a reduction from COF to the sensitivity problem, proving that the latter is $\Sigma^0_3$-hard. Combined with our earlier upper bound, this shows that the sensitivity problem for $d>1$ dimensional cellular automata is $\Sigma^0_3$-complete.

\section{Application}\label{sec:twin}

In this section, we explore an application of our results in number theory, specifically in relation to the twin prime conjecture. While we do not believe this to be a suitable approach for solving the twin prime conjecture, we argue that it demonstrates a concrete application of the arithmetical hierarchy that is not often explored. Furthermore, this construction illustrates how complexity classifications can provide unexpected connections between different areas of mathematics.

\begin{proposition}
The twin prime conjecture is a $\Pi^0_2$ statement in the arithmetical hierarchy.
\end{proposition}

\begin{proof}
The twin prime conjecture can be stated as:
\begin{equation}
\forall n, \exists p (p > n \land \text{Prime}(p) \land \text{Prime}(p+2))
\end{equation}
\qed
\end{proof}

We construct a Turing machine $M$ that, given an input $n$ in unary (represented by a string of 0s), halts if and only if it finds a pair of twin primes larger than or equal to $n$.


\begin{theorem}
There exists a cellular automaton $G_e$ that is sensitive to initial conditions if and only if the twin prime conjecture is true.
\end{theorem}

\begin{proof}
By our construction, the Turing machine $M$ halts on all inputs if and only if the twin prime conjecture is true. Using our reduction from the TOT problem to the sensitivity problem for cellular automata, we can construct a cellular automaton $G_e$ that is sensitive to initial conditions if and only if $M$ halts on all inputs. Therefore, $G_e$ is sensitive to initial conditions if and only if the twin prime conjecture is true. \qed
\end{proof}

\begin{remark}
The construction used in Theorem~\ref{thm:dim1} can be adapted to any \(\Pi^0_2\) statement. For example, one could replace the twin prime conjecture with the Riemann Hypothesis or Goldbach's conjecture — although these are actually \(\Pi^0_1\) statements and therefore technically simpler, the method still applies. More interestingly, it also applies to statements like $P = NP$, which is known to be \(\Pi^0_2\) \cite{madore2016}. However, in this article we focus on a single illustrative example.
\end{remark}


We have implemented this cellular automaton, which can be explored interactively at the following URL: \url{https://tom-favereau.github.io/misc/ca.html}. The source code for the automaton is also available on GitHub at: \\
\url{https://github.com/tom-favereau/twin_prime_automaton}.

\section{Conclusion and Future Work}

In this paper, we have established the precise complexity of determining sensitivity to initial conditions for cellular automata within the arithmetical hierarchy. Specifically, we have shown that this problem is $\Pi^0_2$-complete for one-dimensional cellular automata and $\Sigma^0_3$-complete for cellular automata of dimension two and higher. These results provide a complete characterization of the complexity of the sensitivity problem across all dimensions.

Additionally, we have provided a new proof of Sutner's result stating that the problem of determining finite nilpotency for one-dimensional cellular automata is $\Pi^0_2$-complete. We have also constructed a cellular automaton that is sensitive to initial conditions if and only if the twin prime conjecture is true.

As a corollary to our main results, we can conclude that the problem of determining non-sensitivity is $\Sigma^0_2$-complete for one-dimensional cellular automata and $\Pi^0_3$-complete for higher dimensions. This complementary result completes our analysis of K{\r{u}}rka's classes.

However, it is important to note that our reductions are not reversible. Consequently, the complexity of the sensitivity problem for reversible cellular automata remains an open question. This presents an interesting avenue for future research, as reversible cellular automata form an important subclass with unique properties and applications.

Furthermore, our work does not address the case of expansive cellular automata, which constitute the final class in K{\r{u}}rka's classification. For expansive cellular automata, we know they are at the level $\Sigma^0_1$ of the arithmetical hierarchy, yet the question of the decidability or undecidability of the problem is a well-known open problem in the field. This gap in our understanding presents another significant direction for future investigations.

Additionally, the complexity of determining the existence of equicontinuity points in cellular automata remains unknown, and it is possible that it might be analytical. This presents another important direction for future research, as having equicontinuity points can be seen as a two-dimensional K{\r{u}}rka class.

\begin{question} What is the complexity of determining whether a $d$-dimensional reversible cellular automaton is sensitive to initial conditions? \end{question}

\begin{question} What is the complexity of determining the existence of equicontinuity points for a $d$-dimensional cellular automaton? \end{question}

\begin{question} What is the complexity of determining whether a one-dimensional cellular automaton is expansive? \end{question}

In conclusion, while our results provide a comprehensive complexity analysis for the sensitivity problem for general cellular automata, they also highlight important open questions and future directions.

\bibliographystyle{splncs04}  
\bibliography{ref}  

\end{document}